\documentclass [10pt,oneside]{amsart}
\usepackage{amsmath}

\usepackage {amsfonts}
\usepackage {amsmath}
\usepackage {amsthm}
\usepackage {amssymb}
\usepackage {framed}
\usepackage {amsxtra}
\usepackage {enumerate}
\usepackage {graphicx}
\usepackage{color}
\usepackage{graphicx}
\usepackage{graphicx}
\usepackage{wrapfig}
\usepackage{lscape}
\usepackage{rotating}
\usepackage{epstopdf}
\usepackage{pdflscape}
\usepackage{setspace}
\usepackage{adjustbox}
\usepackage{caption}

\usepackage{pgf,tikz,pgfplots}
\pgfplotsset{compat=1.17}
\usepackage{mathrsfs}
\usetikzlibrary{arrows}
\definecolor{uuuuuu}{rgb}{0.26666666666666666,0.26666666666666666,0.26666666666666666}
\definecolor{ududff}{rgb}{0.30196078431372547,0.30196078431372547,1.}
\definecolor{ffqqqq}{rgb}{1.,0.,0.}
\definecolor{ududff}{rgb}{0.30196078431372547,0.30196078431372547,1.}
\definecolor{ttttff}{rgb}{0.2,0.2,1.}
\usepackage{tcolorbox}

\makeatletter

\theoremstyle{definition}
\newtheorem{df}{Definition} [section]

\theoremstyle{plain}
\newtheorem{thm}[df]{Theorem}

\newtheorem{lemma}[df]{Lemma}

\newtheorem{obs}[df]{Observation}

\newtheorem{conj}[df]{Conjecture}

\usepackage{xcolor}

\title{On the smallest area $(n-1)$-gon containing a convex $n$-gon}

\author{Elliot Hong}
\address{Riverside High School, Leesburg, VA 20176}
\email{einseoul0@gmail.com}

\author{Dan Ismailescu}
\address{Mathematics Department, Hofstra University, Hempstead, NY 11549.}
\email{dan.p.ismailescu@hofstra.edu}

\author{Alex Kwak}
\address{Avon Old Farms School, Avon, CT 06001}
\email{kwakk@avonoldfarms.com}

\author{Grace Yeeun Park}
\address{Palisades Park High School, Palisades Park, NJ 07650}
\email{graceypark04@gmail.com}

\begin{document}

\maketitle




\bigskip

\begin{tcolorbox}
Approximation of convex disks by inscribed and circumscribed polygons is a classical geometric problem whose study is motivated by various applications in robotics and computer aided design.

This article considers the following optimization problem: given integers $3\le n\le m-1$, find the value or an estimate of
\begin{equation*}
r(n,m)=\max_{P\in \mathcal {P}_m}\,\, \min_{Q\in \mathcal {P}_n,\,Q \supseteq P}
\frac{|Q|}{|P|}
\end{equation*}
where $P$ varies in the set $\mathcal {P}_m$ of all convex $m$-gons, and, for a fixed $m$-gon $P$, the minimum is taken over all $n$-gons $Q$ containing $P$; here $|\cdot|$ denotes area.

It has been proved that $r(3,4)=2$, and this is the only exact value currently known.
In this paper we prove that every unit area convex pentagon is contained in a convex quadrilateral of area no greater than $3/\sqrt{5}$, and that every unit area convex hexagon is contained in a convex pentagon of area no greater than $7/6$. Both results are tight as the case of the regular pentagon (hexagon) shows. In other words, $r(4,5)=3/\sqrt{5}$ and $r(5,6)=7/6$. We propose a conjecture regarding the value of $r(n-1,n)$ for $n\ge 6$.
\end{tcolorbox}

\section{\bf Introduction: definitions, notations and background}

A \emph{convex disk} is a compact convex subset of the plane with nonempty interior; let $\mathcal{K}$ be the class of all convex disks.

A \emph{convex $n$-gon} is a convex disk whose interior is the intersection of $n$ half-planes. This definition allows a given $n$-gon to have anywhere between 3 and $n$ vertices. If an $n$-gon has exactly $n$ vertices we will call it a \emph{proper} $n$-gon.
Let $\mathcal{P}_n$ denote the subclass of $\mathcal{K}$ consisting of all convex $n$-gons.

The area of a (measurable) subset $S$ of the plane is denoted by $|S|$.

Given a convex disk $K$ and a positive integer $n \ge 3$, let $C_n(K)$ denote a minimum area $n$-gon which contains $K$.
The existence of $C_n(K)$ follows from Blaschke's selection theorem \cite{blaschke}. We also note that $C_n(K)$ may not be unique; for instance,  if $K$ is a circular disk then there are infinitely many minimum area (regular) $n$-gons containing $K$.
It should be pointed out that $C_n(K)$ may not necessarily be a proper $n$-gon. For example, if $K$ is a triangle, then $C_n(K) = K$ for every $n \ge 3$.

Given  a positive integer $n\ge 3$ and a convex disk $K$, we consider the ratio between the area of $C_n(K)$ and the area of $K$.
\begin{equation*}
r(n, K):= \frac{|C_n(K)|}{|K|}.
\end{equation*}
It is obvious that $r(n,K)\ge 1$ but how large can it be? We will be interested in the following quantity
\begin{equation}\label{R_n}
R_n:=\sup_{K\in \mathcal{K}} \,r(n,K)=\max_{K\in \mathcal{K}}\, \frac{|C_n(K)|}{|K|}.
\end{equation}

The reason the maximum is attained follows from a result of Macbeath \cite{macbeath} who proved that the space of
affine equivalence classes of convex regions is compact.
This motivates the problem of finding the convex regions $K$ that produce the extreme values of $R_n$

It is known that $R_3=2$, the parallelogram being the extremal convex disk. This was already proved by Gross \cite{Gro18} in 1918. Alternative proofs were given by Eggleston \cite{Egg53} and  Chakerian \cite{Cha73}.

Chakerian and Lange \cite{CL71} proved that $R_4\le \sqrt{2}=1.414\ldots$; the equality sign was later removed by Kuperberg in \cite{Kup83}. Kuperberg also conjectured that $R_4=3/\sqrt{5}=1.341\ldots$, with the regular pentagon being the extremal convex disk. Note that there is a sizeable gap between the conjectured value of $R_4$ and the current upper estimate.

Ironically enough, we have better estimates for large values of $n$.

Chakerian \cite{Cha73} proved that
\begin{equation}\label{chak73}
R_n\le \frac{2\pi}{n}\csc\frac{2\pi}{n},\quad {\text {for all}}\,\, n\ge 3.
\end{equation}
This was improved by Ismailescu \cite{Ism09}, who showed that
\begin{equation}\label{isma}
R_n\le \sec\frac{\pi}{n},\quad {\text {for all}}\,\, n\ge 3.
\end{equation}
For large $n$, the best asymptotic estimate is due to L. Fejes T\'{o}th \cite{FT40}:
\begin{equation}\label{lft}
\frac{n}{\pi}\tan\frac{\pi}{n}\le R_n\le \frac{n-2}{\pi}\tan\frac{\pi}{n-2},\quad {\text {for all}}\,\, n\ge 5.
\end{equation}
The lower bound is obtained when $K$ is a circular disk; in this case, it is well known the minimum area $n$-gon containing $K$ is the regular one.

Given the difficulties in finding any other exact values of $R_n$, it would be fruitful to try restrict the class of convex disks under study.

One approach is to consider only \emph{centrally symmetric} convex disks, that is, convex disks for which $K=-K$. Let $\mathcal{K}^{*}$ denote the subset of $\mathcal{K}$ consisting of all centrally symmetric convex disks.
Analogously to \eqref{R_n} define
\begin{equation}\label{Rnsym}
R^{*}_n:=\sup_{K\in \mathcal{K}^*}\, \frac{|C_n(K)|}{|K|}.
\end{equation}
It is immediate that $R^*_n\le R_n$ for all $n\ge 3$ and consequently, $R^*_3=R_n=2$ with the parallelogram being the extremal disk.
Petty \cite{Pet55} proved that $R^*_4=4/3$; the extremal domain is the regular hexagon. A different proof was given by Pelczinski and Szarek \cite{PS91}.

We remark that the problem of finding the exact value of $R^*_6$ is equivalent to a question raised by Reinhardt \cite{reinhardt} regarding the determination of the centrally symmetric convex disk with lowest packing density. It is known that
\begin{equation*}
R^*_6\ge \frac{2\sqrt{2}-1}{8-4\sqrt{2}-\ln{2}}=1.1081\dots.
\end{equation*}
A long standing conjecture of Reinhardt is that equality holds, and the extremal domain is the \emph{smoothed octagon} - a regular octagon with corners removed along arcs of hyperbolae.
Some recent work of Hales \cite{hales} points towards the veracity of this conjecture.

To the best of our knowledge, there are no additional results involving $R^*_n$.

Rather than limiting the discussion to the class of centrally symmetric convex disks, we may restrict ourselves to considering the class of all convex $m$-gons. We present the details in the next section.

\section{\bf Minimum area polygons containing a given polygon}

For a given convex $m$-gon $P$, and a integer $n$ with $3\le n\le m-1$, let $C_n(P)$ be a minimal area $n$-gon which contains $P$.
The problem of determining $C_n(P)$ is motivated by applications in robotics \cite{chazelle} and computer aided design \cite{doribenbessat}.

Various algorithms have been devised for computing $C_n(P)$, see e. g. \cite{aggarwalchangyap} and the references therein. The most efficient algorithm currently known is due to Aggarwal, Chang and Yap and runs in $O(m \log{m} \log{n})$ time.

Most of the polygon circumscribing algorithms rely on the following result of DePano \cite{depano} (see also \cite{zalgaller}).
\begin{lemma}\label{depano}
Given a convex $m$-gon $P$ and a positive integer $n$ with $3\le n\le m-1$, there always exists a convex $n$-gon $C_n(P)$ of minimal area such that at least $n-1$ of sides of $C_n(P)$ contain sides of $P$. Furthermore, the midpoint of each of the sides of $C_n(P)$ belongs to $P$.
\end{lemma}

Similarly to the quantity $R_n$ introduced in \eqref{R_n}, for every $n, m$ satisfying $3\le n\le m-1$ we define
\begin{equation}\label{rnm}
r(n,m):=\sup_{P\in \mathcal{P}_m} \, \frac{|C_n(P)|}{|P|}.
\end{equation}

It follows from definitions \eqref{R_n} and \eqref{rnm} that $r(n,m)\le R_n$, for all $n, m$ with $3\le n<m$. Moreover, since every convex disk can be approximated arbitrarily close by an appropriately chosen polygon, it follows that
\begin{equation}\label{eq}
R_n=\sup_{m>n} r(n,m), \quad \text{for all}\,\, n\ge 3.
\end{equation}

As computing exact values of $R_n$ is a difficult problem, one may expect that the question of finding exact answers for $r(n,m)$ is equally challenging. While this is indeed the case in general, we address the simpler case when $m=n-1$.

Using Lemma \ref{depano} it can be easily checked that the smallest quadrilateral containing a unit area regular pentagon has area $3/\sqrt{5}$ --- see figure \ref{fig1}(a).

\begin{figure}[ht]
\centering
\includegraphics[width=1\linewidth]{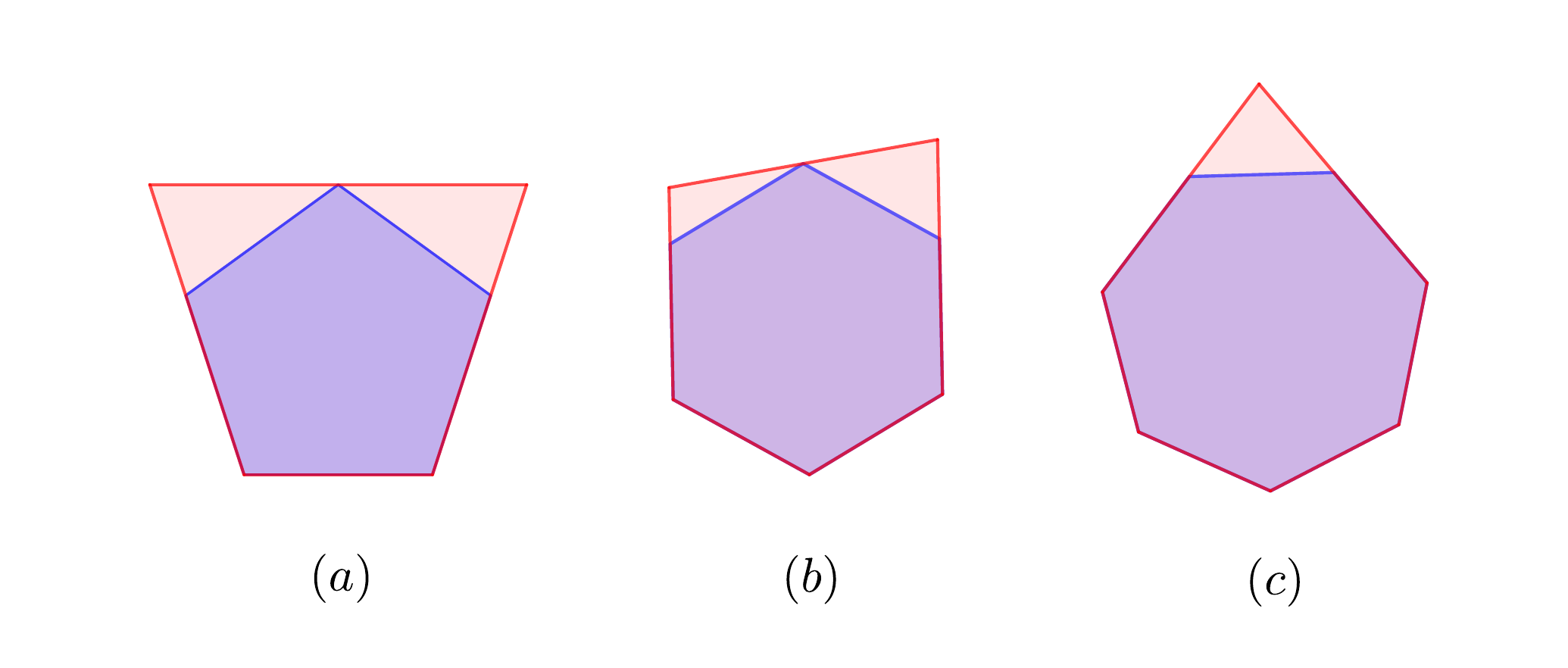}
\vspace{-1cm}
\caption{\small{Minimal area $(n-1)$-gons containing regular $n$-gons}}
\label{fig1}
\end{figure}

Similarly, the smallest pentagon containing a unit area regular hexagon has area $7/6$ as shown in figure \ref{fig1}(b). Note that in this case there are infinitely many optimal circumscribing pentagons. Finally, a unit area regular heptagon is contained in a hexagon of area $1+\tan(2\pi/7)\sec(\pi/7)/7=1.1725\dots$ as shown in figure \ref{fig1}(c).

It is natural to expect that out of all unit area convex $n$-gons, the (affine) regular one requires the largest minimal area circumscribed $(n-1)$-gon.
Thus, we propose the following
\begin{conj}
\begin{equation}
r(4,5)= \frac{3}{\sqrt{5}} \quad \text{and}  \quad r(n-1, n)= 1+\frac{\tan(2\pi/n)}{n\cos(\pi/n)}\quad \text{for all} \,\,  n\ge 6.
\end{equation}
\end{conj}

The main purpose of this paper is to show that $r(4,5)=3/\sqrt{5}$ and $r(5,6)=7/6$, thus confirming the first two cases of the above conjecture.

\subsection{\bf The Main Technique}

In the sequel we will use the \emph{outer product} of two vectors to express areas of various polygons. This
operation, also known as \emph{exterior product}, is defined as follows:

For any two vectors $\mathbf{v}=(a,
\, b)$ and $\mathbf{u}=(c,\,d)$, the outer product of $\mathbf{v}$ and $\mathbf{u}$ be given by
$$\mathbf{v}\wedge \mathbf{u}:=\frac{1}{2}\cdot (ad-bc).$$

It is easy to see that the outer product represents the \emph{signed
area} of the triangle determined by the vectors $\mathbf{v}$ and
$\mathbf{u}$, where the $\pm$ sign depends on whether the angle
between $\mathbf{v}$ and $\mathbf{u}$ - measured in the
counterclockwise direction from $\mathbf{v}$ towards $\mathbf{u}$ -
is smaller than or greater than $180^{\circ}$.

The following properties of the outer product are simple
consequences of the definition and are going to be used extensively
in the remaining part of the paper.
\begin{eqnarray*}
1.&& \mathbf{v}\wedge \mathbf{u} =- \mathbf{u}\wedge\mathbf{v},\quad \text{
anticommutativity. In particular,} \quad \mathbf{v}\wedge\mathbf{v}=0.\\
2.&& (\alpha \mathbf{v}+\beta \mathbf{u})\wedge \mathbf{w}=\alpha (\mathbf{v}\wedge\mathbf{w})
+\beta(\mathbf{u}\wedge\mathbf{w}), \quad \text{linearity}.
\end{eqnarray*}

\section{\bf A small area quadrilateral containing a pentagon}
As mentioned earlier, the unit area regular pentagon requires a circumscribed quadrilateral of area at least $3/\sqrt{5}$. It follows that $r(4,5)\ge 3/\sqrt{5}$. We prove the opposite inequality in the theorem below.

\begin{thm}\label{thmpentagon}
Every convex pentagon $ABCDE$ is contained in a quadrilateral $BCFG$ such that $|BCFG| \le 3/\sqrt{5}\cdot |ABCDE|$.
\end{thm}
\begin{proof}
Let $ABCDE$ be an arbitrary convex pentagon.
After an eventual relabeling of the vertices we may assume that
\begin{equation}\label{ABC}
|DEA|=\min\{|ABC|,\, |BCD|,\, |CDE|,\,
|DEA|,\, |EAB|\}
\end{equation}
In the literature, the triangles formed by three consecutive
vertices of a convex polygon are sometimes called \emph{ears}.
Assumption (\ref {ABC}) above fixes the ear of least area. Denote
the intersection of $AC$ and $BD$ by $O$, and define $\mathbf{u}
:= \overrightarrow{OD}$, $\mathbf{v} := \overrightarrow{OA}$.

After an appropriate scaling, we may assume that $\mathbf{u}\wedge
\mathbf{v} = |AOD| = 1$. Since $A$, $O$, and $C$ are
collinear and $D$, $O$, and $B$ are collinear, we can write
$\overrightarrow{BO} = a \cdot \overrightarrow{OD} = a \mathbf{u}$
and $\overrightarrow{CO} = b \cdot \overrightarrow{OA} = b
\mathbf{v}$, for some scalars $a,\ b > 0$ --- see figure \ref{fig2}.

\begin{figure}[ht]
\centering
\includegraphics[width=0.9\linewidth]{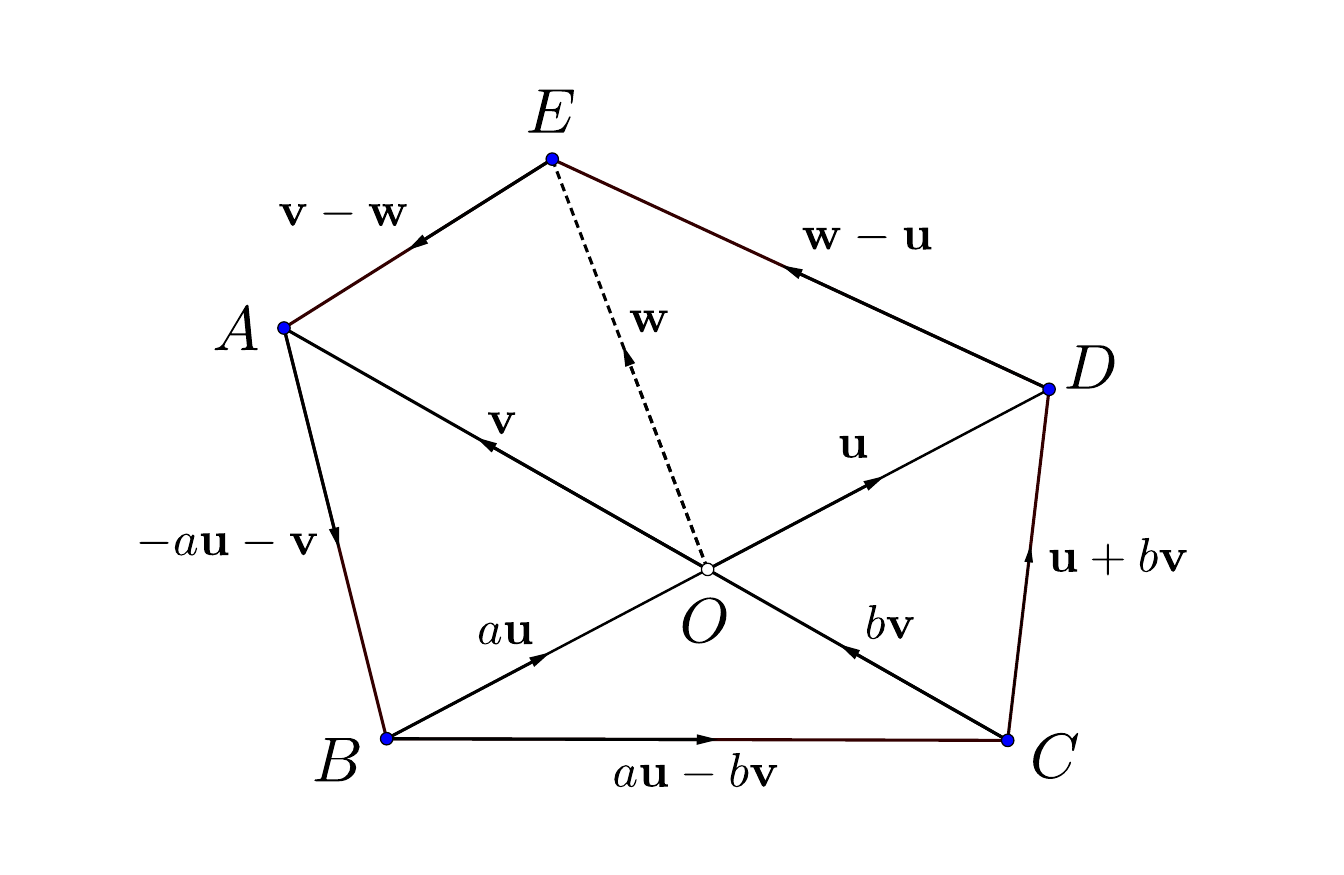}
\vspace{-0.5cm}
\caption{\small{Defining a convex pentagon}}
\label{fig2}
\end{figure}

\noindent Using the triangle rule, we obtain that
$\overrightarrow{AB} = -a\mathbf{u}-\mathbf{v}$,
$\overrightarrow{BC} = a\mathbf{u} - b\mathbf{v}$, and
$\overrightarrow{CD} = \mathbf{u}+b\mathbf{v} $.

Since any vector can be written as a linear combination of two given independent vectors,
let $\overrightarrow{OE} = \mathbf{w} := c \mathbf{u} + d \mathbf{v}$, with $c, d>0$.
It follows that $\overrightarrow{DE} = \mathbf{w} - \mathbf{u}$
and $\overrightarrow{AE} = \mathbf{v} - \mathbf{w}$. We have
\begin{eqnarray*}
|ODE| &=& \mathbf{u}\wedge \mathbf{w} =
\mathbf{u}\wedge (c\mathbf{u} + d\mathbf{v}) = c(\mathbf{u}\wedge \mathbf{u}) + d (\mathbf{u}\wedge\mathbf{v}) = d,\\
|OAE|&=& \mathbf{w}\wedge \mathbf{v} =
(c\mathbf{u} + d\mathbf{v})\wedge \mathbf{v} = c(\mathbf{u}\wedge \mathbf{v}) + d (\mathbf{v}\wedge \mathbf{v}) = c.
\end{eqnarray*}

After similar calculations, we can write the areas of various
triangles in pentagon $ABCDE$ in terms of the positive constants
$a,\,b,\,c,\,d$ as shown below:
\begin{equation*}
|OAB|= -\mathbf{v}\wedge a\mathbf{u} = a,\,|OBC| = -a\mathbf{u}\wedge -b\mathbf{v} =ab,\,|OCD| = -b\mathbf{v}\wedge (b\mathbf{v}+ \mathbf{u}) =b.
\end{equation*}

We can now express the total area of the pentagon in terms of the parameters $a, b, c$, and $d$.

$|ABCDE| = |OAB| + |OBC| + |OCD| +
|ODE| + |OEA|$, that is,
\begin{equation}\label{ABCDE}
|ABCDE| = a+b+c+d+ab.
\end{equation}

Next, we compute the areas of the ears of the pentagon.
\begin{eqnarray}\label{laterlemma}
|ABC| &=& \overrightarrow{AB} \wedge \overrightarrow{BC}= (-a\mathbf{u}-\mathbf{v})\wedge (a\mathbf{u} - b\mathbf{v}) = a + ab,\notag\\
|BCD| &=& \overrightarrow{BC} \wedge \overrightarrow{CD}= (a\mathbf{u}- b\mathbf{v})\wedge (b\mathbf{v} + \mathbf{u}) = b + ba,\notag\\
|CDE| &=& \overrightarrow{CD} \wedge \overrightarrow{DE}= (b\mathbf{v}+\mathbf{u})\wedge ((c-1)\mathbf{u} + d\mathbf{v}) = b + d - bc,\\
|DEA| &=& \overrightarrow{DE} \wedge \overrightarrow{EA}= ((c-1)\mathbf{u}+ d\mathbf{v})\wedge (-c\mathbf{u} + (1-d)\mathbf{v}) = c +d -1,\notag\\
|EAB| &=& \overrightarrow{EA} \wedge \overrightarrow{AB}= (-c\mathbf{u}+ (1-d)\mathbf{v})\wedge (-a\mathbf{u} - \mathbf{v}) = a + c  -ad.\notag
\end{eqnarray}
Translating assumption (\ref{ABC}) in terms of $a$, $b$, $c$, and $d$, we obtain
\begin{eqnarray*}
|DEA| &\leq& |ABC|\iff c + d - 1 \leq a + ab
\iff c + d \leq 1+a+ab,\\
|DEA| &\leq& |BCD|\iff c + d - 1 \leq b + ab
\iff c + d \leq 1+b+ab,\\
|DEA| &\leq& |CDE| \iff c+d-1 \leq b+d-
bc\iff c\leq 1,\\
|DEA| &\leq& |EAB| \iff c+d-1 \leq a + c -
ad\iff d\leq 1.
\end{eqnarray*}
We introduce two more notations
\begin{equation}\label{ef}
e:=1+a+ab-c-d \quad \text{and}\quad f:=1+b+ab-c-d.
\end{equation}

Then, the above inequalities can be summarized as
\begin{equation}\label{cdef}
1-c\ge 0, \,\,1-d\ge 0,\,\,e\ge 0,\,\,\text{and}\,\, f\ge 0.
\end{equation}

Let us assess the situation for a moment. At this point, we have relatively simple expressions for the
area of the pentagon $ABCDE$ and for the areas of its ears, in terms of the variables $a$, $b$, $c$ and $d$.
Notice that $a$, $b$, $c$, and $d$ are strictly positive but not completely independent as the conditions above reflect.

What are the values of $a$, $b$, $c$, and $d$ if $ABCDE$ is an affine regular pentagon? It is very easy to see that in this case $ODEA$ is a parallelogram, so necessarily $\mathbf{w}=\mathbf{u}+\mathbf{v}$, which means that $c=d=1$.
On the other hand, a simple trigonometry exercise shows that $a=b=\frac{1}{2}\sec(\pi/5)=(\sqrt{5}-1)/2$.
We record this observation for future reference.
\begin{obs}\label{extremum}
The pentagon $ABCDE$ is affine regular if and only if $a=b=\phi\overset{def}{=}(\sqrt{5}-1)/2$ and $c=d=1$.
\end{obs}

Next we construct a quadrilateral containing $ABCDE$ as follows: 

through vertex $E$ construct the parallel line to the diagonal $AD$. 
The extensions of sides $CD$ and $BA$ intersect this line at points $F$ and $G$ - see figure \ref{fig3}.
\begin{figure}[ht]
\centering
\includegraphics[width=1\linewidth]{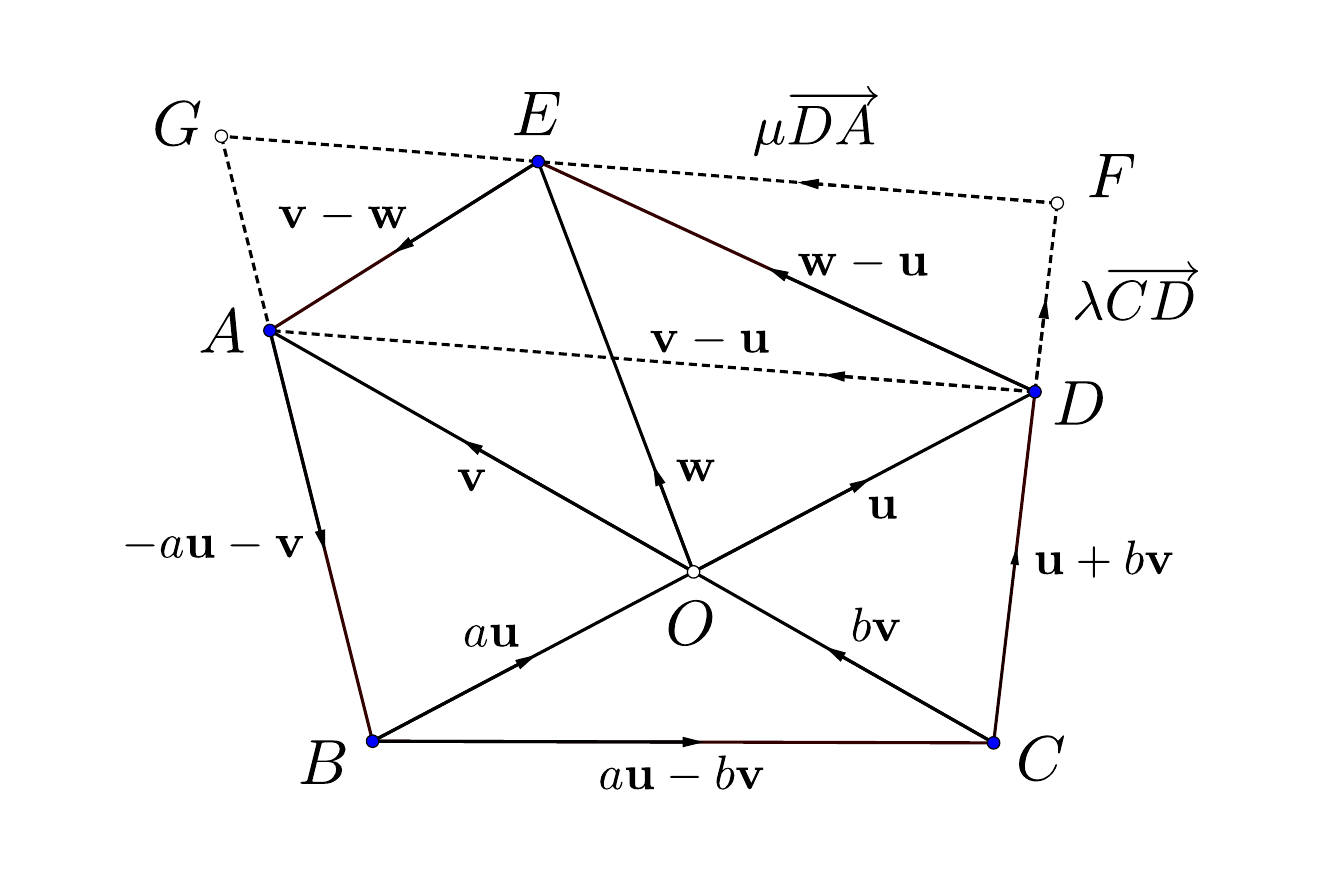}
\vspace{-0.5cm}
\caption{\small{A small area quadrilateral containing the pentagon}}
\label{fig3}
\end{figure}

To prove Theorem \ref{thmpentagon} it would suffice to prove that
\begin{equation}\label{wts}
\frac{|BCFG|}{|ABCDE|}\le \frac{3}{\sqrt{5}}.
\end{equation}
Let us express the areas of triangles $DFE$ and $EGA$ in terms of $a$, $b$, $c$, and $d$.

There exist positive constants $\lambda$ and $\mu$ such that
$\overrightarrow{DF}=\lambda\overrightarrow{CD}$ and $\overrightarrow{FE}=\mu\overrightarrow{DA}$.
Since $\overrightarrow{DF}+\overrightarrow{FE}=\overrightarrow{DE}$, it follows that
$\lambda\overrightarrow{CD}+\mu\overrightarrow{DA}=\overrightarrow{DE}$.

Taking outer product with $\overrightarrow{DA}$ on both sides, we obtain
\begin{align*}
&(\lambda\overrightarrow{CD}+\mu\overrightarrow{DA})\wedge \overrightarrow{DA}=\overrightarrow{DE} \wedge\overrightarrow{DA}
\implies \lambda(\overrightarrow{CD}\wedge\overrightarrow{DA})=\overrightarrow{DE}\wedge\overrightarrow{DA}\implies\\
&\implies \lambda|CDA|=|DEA|\implies \lambda=\frac{c+d-1}{1+b}.\quad \text{Hence}
\end{align*}
\begin{equation}\label{DFE}
|DFE|=\overrightarrow{DF}\wedge\overrightarrow{DE}=\lambda(\overrightarrow{CD}\wedge\overrightarrow{DE})=
\lambda|CDE|=\frac{(c+d-1)(b-bc+d)}{1+b}.
\end{equation}
Similar computations show that the area of $EGA$ can be written as
\begin{equation}\label{EGA}
|EGA|=\frac{(c+d-1)(a-ad+c)}{1+a}.
\end{equation}
Since we intend to show that inequality (\ref{wts}) holds it would suffice to prove that
\begin{equation}
|DFE|+|EGA|\le \left(\frac{3}{\sqrt{5}}-1\right)\cdot|ABCDE|.
\end{equation}
Using now (\ref{ABCDE}), (\ref{DFE}), and (\ref{EGA}), everything reduces to proving that
\begin{equation*}
\frac{(c+d-1)(a-ad+c)}{1+a}+\frac{(c+d-1)(b-bc+d)}{1+b}\le \left(\frac{3}{\sqrt{5}}-1\right)(a+b+c+d+ab).
\end{equation*}

After clearing the denominators, multiplying by $2$, and rearranging, the inequality to prove is equivalent to
\begin{align}\label{fnonneg}
g(a,b,c,d)&:=(6\sqrt{5} - 10)(a^2\,b^2 + 2\,a^2\,b + 2\,a\,b^2 + a^2 + b^2) + 20\,a\,b\,c\,d + \notag\\
&+10\,a\,b\,(c^2 + d^2)+ (18\sqrt{5} - 10)\,a\,b + 6\sqrt{5}\,(a + b + c + d) - \notag\\
&-(20 - 6\sqrt{5})(a + b)(c + d) -(40 - 6\sqrt{5})\,a\,b\,(c + d) - 10\,(c + d)^2 \ge 0.
\end{align}

Recall that $a$, $b$, $c$ and $d$ are positive; in addition, we have the side constraints (\ref{cdef}). Proving
inequality (\ref{fnonneg}) via the standard techniques is a very challenging task. Instead, we use an alternate approach. We know that $g(a,b,c,d)=0$ if $a=b=\phi=(\sqrt{5}-1)/2, c=d=1$, and we want to show that this is the
only case when equality happens.

The main idea is simple: \emph{try to write $g(a,b,c,d)$ as a sum of nonnegative terms}.

More specifically, we attempt to express $g(a,b,c,d)$ as
\begin{equation}\label{pos}
g(a,b,c,d)=\sum_J p_J\,a^{i_1}b^{i_2}c^{i_3}d^{i_4}e^{i_5}f^{i_6}(1-c)^{i_7}(1-d)^{i_8}(a-\phi)^{2i_9}(b-\phi)^{2i_{10}},
\end{equation}
with all coefficients $p_J\ge 0$ where $J=\{i_1,i_2,i_3,i_4,i_5,i_6,i_7,i_8,i_9, i_{10}\}$.

The exponents $i_k$ are nonnegative integers for all $1\le k \le 10$.  Recall that $a, b, c, d$ are positive while $1-c, 1-d, e$, and $f$ are nonnegative by \eqref{cdef}.

The problem reduces to solving a linear system whose equations are obtained by matching the coefficients of the terms $a^ib^jc^kd^l$ from
(\ref{fnonneg}) to those in (\ref{pos}).

We want this system to have nonnegative solutions. We used MAPLE to find such a solution. It turns out
that $g(a,b,c,d)$ can be rewritten as
\begin{align*}
&\phantom{+i}(60 - 26\sqrt{5})\,(c\,f + d\,e) + (21 - 7\sqrt{5})\,(a\,f + b\,e) + (16\sqrt{5} - 30)\,c\,d\,(e + f) + \\
&+(80 - 32\sqrt{5})\,a\,b\,(1 - c)\,(1 - d) + (3\sqrt{5} - 5)\,(b^2\,(a - \phi)^2 + a^2\,(b - \phi)^2) + \\
&+(5 - \sqrt{5})\,(b^2(1 - c)^2 + a^2(1 - d)^2) + (10 - 2\sqrt{5})\,(b^2c\,(1 - c) + a^2d\,(1 - d)) +\\
&+ 10\,a\,b((1 - c)^2 +(1 - d)^2) +(48 - 20\sqrt{5})\,(1-c)(1-d)(a\,(1 - c)+b\,(1 - d)) \\
&+ (10\sqrt{5} - 20)\,(d\,(1 - c)^2 + c\,(1 - d)^2) + (5 - \sqrt{5})\,(d^2\,(a - \phi)^2 + c^2\,(b - \phi)^2) + \\
&+(11\sqrt{5} - 21)\,(b\,(a - \phi)^2 + a\,(b - \phi)^2) + (6\sqrt{5} - 10)\,(b\,(1 - c)^2 + a\,(1 - d)^2) +\\
&+(97 - 41\sqrt{5})\,(1 - c)\,(1 - d)\,(a\,c+b\,d) + (17\sqrt{5} - 31)\,(b\,c\,(1 - d) + a\,d\,(1 - c))+\\
&+(49 - 21\sqrt{5})\,(b\,d^2\,(1 - c) + a\,c^2\,(1 - d)) +(22\sqrt{5} - 40)\,(1 - c)\,(1 - d)\,(c+d).
\end{align*}
It is straightforward to check that all terms of the above sum are nonnegative, and that $g(a,b,c,d)=0$ only when $a=b=\phi$ and $c=d=1$, as desired.
The proof of Theorem \ref{thmpentagon} is complete.
\end{proof}

\section{\bf A small area pentagon containing a hexagon}
\begin{thm}\label{thmhexagon}
Every convex hexagon $ABCDEF$ is contained in a pentagon $BCDGH$ such that $|BCDGH| \le 7/6\cdot |ABCDEF|$.
\end{thm}

\begin{proof}
Let $ABCDEF$ be an arbitrary convex hexagon. Suppose that the long
diagonals, $AD$, $BE$, and $CF$ are not concurrent. If these
diagonals do have a common point, then perturb the position of one
of the vertices by an arbitrarily small amount so that the diagonals
are not concurrent anymore. By continuity, any inequality which is
valid in the latter case is also valid in the former.
Let $M=AD\cap BE,\,N=AD\cap CF,\,$ and $P=CF\cap BE$. Denote
$\mathbf{u}=\overrightarrow{MN},
\mathbf{v}=\overrightarrow{MP}$ as shown in figure \ref{fig4}.

It follows that
$\overrightarrow{NP}=\mathbf{w}:=\mathbf{v}-\mathbf{u}$.
After an appropriate scaling we may assume that $|MNP|=
\mathbf{u}\wedge\mathbf{v} = \mathbf{u}\wedge\mathbf{w} = \mathbf{v}\wedge\mathbf{w} =1$.

Since $A,\,M,\,N,\,D$ are collinear, there exist positive scalars $a$ and $d$ such that $\overrightarrow{AM}=a\mathbf{u},
\overrightarrow{ND}=d\mathbf{u}$.

Similarly, there exist positive scalars $b,\,c,\,e,\,f$ so that
$\overrightarrow{BM}=b\mathbf{v},
\overrightarrow{CN}=c\mathbf{w},
\overrightarrow{PE}=e\mathbf{v},
\overrightarrow{PF}=f\mathbf{w}$. Without loss of generality we may assume that
\begin{equation}\label{assumeamin}
a=\min \{a,\,b,\,c,\,d,\,e,\,f\}.
\end{equation}
\begin{figure}[ht]
\centering
\includegraphics[width=0.9\linewidth]{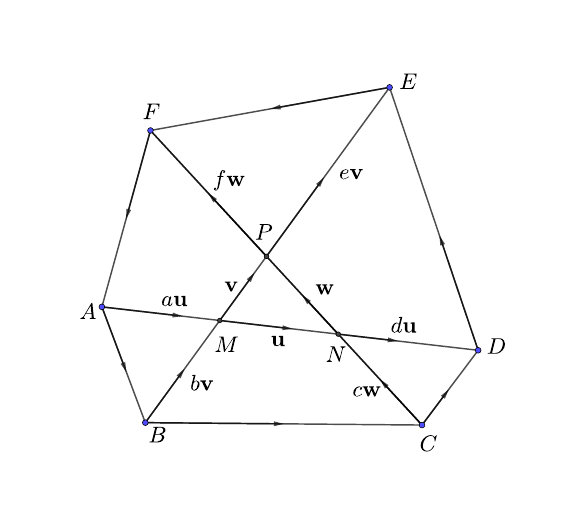}
\vspace{-0.5cm}
\caption{\small{Defining a convex hexagon}}
\label{fig4}
\end{figure}
Using outer products, we obtain the following area formulas.
\begin{align*}
|ANF|&=(a+1)\mathbf{u}\wedge (f+1)\mathbf{w}=(a+1)(f+1).\\
|BPC|&=(b+1)\mathbf{v}\wedge
(c+1)(\mathbf{v}-\mathbf{u})=(b+1)(c+1).\\
|DME|&=(d+1)\mathbf{u}\wedge (e+1)(\mathbf{v})=(d+1)(e+1).\\
|AMB|&= a\mathbf{u}\wedge b\mathbf{v}=ab.\\
|CND|&= d\mathbf{u}\wedge c\mathbf{w}=cd.\\
|EPF|&= e\mathbf{v}\wedge f\mathbf{w}=ef.
\end{align*}
We are now in position to compute the area of the hexagon $ABCDEF$.
\begin{align*}
|ABCDEF|&=|ANF|+|BPC|+|DME|+|AMB|+\\
&+|CND|+|EPF|-2|MNP|,\quad {\text{which implies
that}}
\end{align*}
\begin{equation}\label{hex}
|ABCDEF|=1+a+b+c+d+e+f+ab+bc+cd+de+ef+fa.
\end{equation}
Since $c\ge a, d\ge a$, and $e\ge a$ it follows that
\begin{equation}\label{boundhex}
|ABCDEF|> 4a+b+f+2a^2+2ab+2af.
\end{equation}
We express the areas of triangles $EFA$, $FAB$, $ABC$, $FBC$, and $EFB$.
\begin{align}\label{fivetriangles}
|EFA|&=\overrightarrow{AF}\wedge \overrightarrow{EF}=(a\mathbf{u}+\mathbf{v}+f\mathbf{w})\wedge (-e\mathbf{v}+f\mathbf{w})=f(1+e+a)-ea,\notag\\
|FAB|&=\overrightarrow{AB}\wedge \overrightarrow{AF}=(a\mathbf{u}-b\mathbf{v})\wedge (a\mathbf{u}+\mathbf{v}+\mathbf{w})=a(1+f+b)-fb,\notag\\
|ABC|&=\overrightarrow{AB}\wedge \overrightarrow{BC}=(a\mathbf{u}-b\mathbf{v})\wedge (\mathbf{u}+b\mathbf{v}-c\mathbf{w})=b(1+a+c)-ac,\\
|FBC|&=\overrightarrow{BC}\wedge \overrightarrow{CF}=(\mathbf{u}+b\mathbf{v}-c\mathbf{w})\wedge (1+c+f)\mathbf{w}=(1+b)(1+c+f),\notag\\
|EFB|&=\overrightarrow{BE}\wedge \overrightarrow{EF}=(1+b+e)\mathbf{v}\wedge (-e\mathbf{v}+f\mathbf{w})=f(1+b+e).\notag
\end{align}

Next we construct a pentagon which contains the hexagon $ABCDEF$ as follows: through vertex $A$ construct the parallel line to the diagonal $FB$. The extensions of sides $EF$ and $CB$ intersect this line at points $G$ and $H$ as shown in figure \ref{fig4}.

\begin{figure}[ht]
\centering
\includegraphics[width=0.9\linewidth]{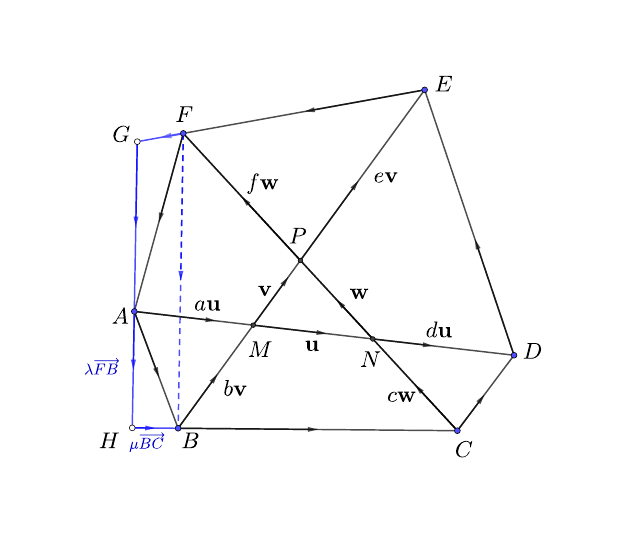}
\vspace{-0.5cm}
\caption{\small{A small area pentagon containing a hexagon}}
\label{fig5}
\end{figure}

To prove Theorem \ref{thmhexagon} it would suffice to show that
\begin{equation}\label{wtshex}
\frac{|BCDGH|}{|ABCDEF|}\le \frac{7}{6}, \quad \text{or equivalently}\quad |ABCDEF|\ge 6|AHB|+6|FGA|.
\end{equation}
We write the areas of triangles $AHB$ and $FGA$ in terms of $a, b, c, d, e$, and $f$.

There exist positive constants $\lambda$ and $\mu$ such that
$\overrightarrow{AH}=\lambda\overrightarrow{FB}$ and $\overrightarrow{HB}=\mu\overrightarrow{BC}$.
Since $\overrightarrow{AH}+\overrightarrow{HB}=\overrightarrow{AB}$, it follows that
$\lambda\overrightarrow{FB}+\mu\overrightarrow{BC}=\overrightarrow{AB}$.

Taking outer product with $\overrightarrow{BC}$ on both sides, we obtain
\begin{align*}
&(\lambda\overrightarrow{FB}+\mu\overrightarrow{BC})\wedge \overrightarrow{BC}=\overrightarrow{AB} \wedge\overrightarrow{BC}
\implies \lambda(\overrightarrow{FB}\wedge\overrightarrow{BC})=\overrightarrow{AB}\wedge\overrightarrow{BC}\implies\\
&\implies \lambda|FBC|=|ABC|\implies \lambda=\frac{|ABC|}{|FBC|}.
\end{align*}

It follows that
\begin{equation}\label{AHB}
|AHB|=\overrightarrow{AH}\wedge\overrightarrow{AB}=\lambda(\overrightarrow{FB}\wedge\overrightarrow{AB})=
\lambda|FAB|=\frac{|FAB|\cdot|ABC|}{|FBC|}.
\end{equation}
Similar computations show that the area of $FGA$ can be written as
\begin{equation}\label{FGA}
|FGA|=\frac{|EFA|\cdot|FAB|}{|EFB|}.
\end{equation}
Using \eqref{fivetriangles} and \eqref{AHB} it follows that
\begin{equation*}
|AHB|=\frac{(a+af+ab-fb)(b+ab+bc-ac)}{(1+b)(1+c+f)}.
\end{equation*}
Differentiating with respect to $c$ we obtain
\begin{equation*}
\frac{\partial |AHB|}{\partial c}=-\frac{(a+af+ab-bf)^2}{(1+b)(1+c+f)^2}\le 0,
\end{equation*}
which implies that $|AHB|$ is decreasing as a function of $c$. Since $c\ge a$ by \eqref{assumeamin} it follows that
\begin{equation}\label{ineqAHB}
|AHB|\le\frac{(a+af+ab-fb)(b+2ab-a^2)}{(1+b)(1+a+f)}.
\end{equation}
Likewise, using \eqref{fivetriangles} and \eqref{FGA} it follows that
\begin{equation*}
|FGA|=\frac{(a+af+ab-fb)(f+ef+af-ae)}{f(1+b+e)}.
\end{equation*}
Differentiating with respect to $e$ we obtain
\begin{equation*}
\frac{\partial |FGA|}{\partial e}=-\frac{(a+af+ab-bf)^2}{b(1+b+e)^2}\le 0,
\end{equation*}
which implies that $|FGA|$ is decreasing as a function of $e$. Since $e\ge a$ by \eqref{assumeamin} it follows that
\begin{equation}\label{ineqFGA}
|FGA|\le\frac{(a+af+ab-fb)(f+2af-a^2)}{f(1+a+b)}.
\end{equation}
Combining now inequalities \eqref{boundhex}, \eqref{ineqAHB} and \eqref{FGA} it follows that
\begin{align}\label{final}
|ABCDEF|-&6|AHB|-6|FGA|> 4a+b+f+2a^2+2ab+2af-\\ &-6(a+ab+af-bf)\cdot\left(\frac{b+2ab-a^2}{(1+b)(1+a+f)}+\frac{f+2af-a^2}{f(1+a+b)}\right).
\end{align}

From \eqref{assumeamin} it follows that $b=a+x$ and $f=a+y$ for some $x, y\ge 0$. After clearing the denominator, the right hand side of \eqref{final} becomes
\begin{align*}
&(14a + 7)(x^3y^2+x^2y^3) + (18a^2 + 11a + 1)x^3y + (18a^2 + 25a + 8)xy^3 +\\
+&(78a^2 + 78a + 15)x^2y^2+(4a^3 + 4a^2 + a)x^3 +(4a^3 + 8a^2 + 5a + 1)y^3+\\
+&(70a^3 + 85a^2 + 26a + 2)x^2y+(70a^3 + 125a^2 + 66a + 9)xy^2+\\
+&(12a^4 + 20a^3 + 11a^2 + 2a)x^2 +(12a^4 + 28a^3 + 23a^2 + 8a + 1)y^2+\\
+&(42a^4 + 84a^3 + 58a^2 + 16a + 1)xy+(2a^5 + 8a^4 + 12a^3 + 7a^2 + a)x+\\
&+(2a^5+ 2a^4 + a^2 + a)y.
\end{align*}
Since this is obviously nonnegative, inequality \eqref{final} follows. This concludes the proof of Theorem \ref{thmhexagon}.
\end{proof}

\section{\bf Concluding remarks and Open questions}
In this paper we proved that every unit area convex quadrilateral is contained in a quadrilateral of area $\le 3/\sqrt{5}$ and that every unit area convex hexagon is contained in a pentagon of area no greater than $7/6$. Both results are optimal. The key idea was to encode a convex pentagon/hexagon, in terms of four/six parameters, respectively. It is doubtful that this approach can be generalized to polygons with arbitrarily many vertices. On one hand, for each additional vertex the number of parameters increases by two, thus presumably leading to more complicated calculations. On the other hand, the line arrangement determined by the long diagonals of the polygon may not longer be unique; this in turn would require the analysis of several different cases. New ideas seem to be needed.

\end{document}